\documentclass[11pt]{article}

\usepackage[a4paper,margin=1in]{geometry}
\usepackage{amsmath,amssymb,amsthm,mathtools}
\usepackage{enumitem}
\usepackage{microtype}
\usepackage[numbers,sort&compress]{natbib}

\usepackage[
colorlinks=true,
linkcolor=blue,
citecolor=blue,
urlcolor=blue,
pagebackref=true
]{hyperref}
\usepackage[nameinlink,noabbrev]{cleveref}
\renewcommand*{\backref}[1]{}

\renewcommand*{\backrefalt}[4]{%
	\ifcase #1\relax
	\or
	\space #2%
	\else
	\space #2%
	\fi
}

\newcommand{\affl}[3]{%
	\noindent #1,
	\textsc{#3}\\
	Email: \texttt{#2}\\[1.5pt]
}

\allowdisplaybreaks
\numberwithin{equation}{section}

\newtheorem{theorem}{Theorem}[section]
\newtheorem{lemma}[theorem]{Lemma}

\newtheorem{corollary}[theorem]{Corollary}
\theoremstyle{definition}
\newtheorem{definition}[theorem]{Definition}
\theoremstyle{remark}

\DeclareMathOperator{\tr}{tr}

\newcommand{\one}{\mathbf 1}
\newcommand{\splus}{s^{+}}
\newcommand{\sminus}{s^{-}}
\newcommand{\Aplus}{A_{+}}
\newcommand{\Aminus}{A_{-}}
\newcommand{\E}{\mathbb E}
\newcommand{\Prob}{\mathbb P}
\newcommand{\norm}[1]{\left\lVert #1\right\rVert}

\title{A Vertex-Localized Positive Square-Energy Strengthening\\
	of Tur\'an's Theorem}

\author{Abhay Jayarajan,  M. Rajesh Kannan,
	Shivaramakrishna Pragada,  Rahul Roy}
\date{}

\begin{document}
	
	\maketitle
	
	\begin{abstract}
		Let $G$ be a graph of order $n$ with the adjacency eigenvalues $\lambda_1(G) \geq \dots \geq \lambda_n(G) $. Let $c(v)$ denote the maximum
		order of a clique containing vertex $v$.  We prove the vertex-localized
		positive square-energy inequality
		\[
		\sqrt{\splus(G)}
		\leq 
		\sum_{v\in V}\left(1-\frac1{c(v)}\right),
		\]
		where
		\[
		\splus(G)=\sum_{\lambda_i(G)>0}\lambda_i(G)^2.
		\]
		We also characterize equality. Apart from edgeless graphs, equality holds precisely for graphs obtained
		from a complete regular multipartite graph by adding an arbitrary number
		of isolated vertices. This settles a conjecture of Kannan, Kumar and Pragada. 
		
	\end{abstract}
	
	
	
	\section{Introduction}
	
	Throughout, we assume all graphs are finite and simple.  Let $G=(V(G),E(G))$ be a
	graph of order $n$. Whenever the context is clear, we write $V$ and $ E$ for $V(G)$ and $E(G)$, respectively. For $S\subseteq V$, $G[S]$ denotes the graph induced by vertices in $S$. For $v\in V$, let $N(v)$ denote the open neighborhood of $v$. Let $\omega(G)$ denote the clique number of $G$.
	Let $\mathbf 1$ denote the all-ones vector of the appropriate order.
	For matrices of the same order, let $X\circ Y$ denote their
	Hadamard product.
	We write
	\[
	\|X\|_F^2:=\operatorname{tr}(X^{\top}X)
	\]
	for the squared Frobenius norm. Let the adjacency matrix $A=A(G)$ of the graph $G$ have eigenvalues
	\[
	\lambda_1(G)\geq\lambda_2(G)\geq\cdots\geq\lambda_n(G).
	\]
	The positive and negative square energies of $G$ are defined as follows: 
	\[
	\splus(G):=\sum_{\lambda_i(G)>0}\lambda_i(G)^2,
	\qquad
	\sminus(G):=\sum_{\lambda_i(G)<0}\lambda_i(G)^2.
	\]
	For $v\in V$, let $c(v)$ be the maximum order of a clique of $G$
	containing $v$, that is,
	\[
	c(v):=
	\max\bigl\{|Q|:v\in Q\subseteq V,\ G[Q]\text{ is complete}\bigr\}.
	\]
	We call $c(v)$ the local clique number of $v$.
	Thus $c(v)=1$ exactly when $v$ is isolated.

	The square energies were used by Wocjan and
	Elphick~\cite{MR3104537}, in connection with a spectral lower
	bound for the chromatic number involving the entire adjacency spectrum.
	
	Elphick et al. ~\cite{MR3512335} subsequently initiated the systematic study of
	square energies as graph invariants and conjectured that every connected
	graph $G$ satisfies
	\[
	\min\{\splus(G),\sminus(G)\}\geq n -1.
	\]
	Further structural and extremal work was developed by Abiad et
	al.~\cite{MR4626652}, Akbari et al. \cite{MR4961868, akbari2025refinement}, Elphick and Linz~\cite{MR4755798}, and
	Zhang~\cite{zhang2024extremal}. In \cite{MR4971662}, a connection of $p$-energies to chromatic numbers has been studied. The connected square-energy conjecture was
	proved by Liu et al.~\cite{liu2026positive2}; its equality cases
	were subsequently determined by Hu, Liu, and Wang~\cite{hu2026extremal}.

	Applying the Motzkin-Straus \cite{MR175813} theorem to a nonnegative Perron
	vector, Wilf \cite{MR830598} obtained the following spectral form of Tur\'an's theorem
	\[
	\lambda_1(G)\leq
	\left(1-\frac1{\omega(G)}\right)n.
	\]
	Elphick and Wocjan conjectured that the same right-hand side continues to
	bound $\sqrt{\splus(G)}$; see~\cite{MR4682534, elphick2018conjecturedlowerboundclique}.
	Liu et al. proved the conjecture in \cite{liu2026positive}. Thus, with $r=\omega(G)$,
	\begin{equation}\label{eq:global-LTZ}
		\sqrt{\splus(G)}
		\leq
		\left(1-\frac1r\right)n.
	\end{equation}
	Jayarajan et al.\cite{jayarajan2026equalitycasepositivesquareenergy} subsequently characterized equality
	in~\eqref{eq:global-LTZ}. 
	
	For positive integers $s_1,\ldots,s_p$, let
	$K_{s_1,\ldots,s_p}$ denote the complete $p$-partite graph
	whose partite sets have orders $s_1,\ldots,s_p$.
	
	\begin{theorem} \cite{jayarajan2026equalitycasepositivesquareenergy}\label{thm;globaleq}
		Let $G$ be a graph of order $n$ with clique number $r=\omega(G)$. Then
		\begin{equation}\label{eq:main_result}
			\sqrt{\splus(G)}
			=\left(1-\frac{1}{r}\right)n
		\end{equation}
		if and only if one of the following holds:
		\begin{enumerate}[label=\textup{(\roman*)}]
			\item $r=1$ and $G$ is edgeless;
			\item $r\geq 2$, $r\mid n$, and
			\[
			G\cong K_{\underbrace{n/r,\ldots,n/r}_{r\text{ parts}}}.
			\]
		\end{enumerate}
	\end{theorem}
	A recent direction in extremal combinatorics seeks to replace a global
	parameter or forbidden-substructure condition by quantities attached
	locally to the individual vertices or edges of a graph. Malec and
	Tompkins~\cite{MR4568750} developed this viewpoint systematically
	by deriving localized forms of several classical extremal results,
	including Tur\'an's theorem, the Erd\H{o}s-Gallai theorem, the
	LYM inequality, and the Erd\H{o}s-Ko-Rado theorem. Localized Tur\'an-type inequalities were also obtained independently
	by Brada\v{c}~\cite{Bradac2022}. Adak and
	Chandran~\cite{adak2025vertex} developed a vertex-based localization
	framework for Erd\H{o}s-Gallai-type extremal problems. Adak and Chandran subsequently applied this viewpoint to Tur\'an's theorem by replacing
	the global clique number with the orders of the largest cliques
	containing the individual vertices~\cite{adak2025vertexbasedlocalizationturanstheorem}. In the
	spectral setting, Liu and Ning~\cite{MR4964125} proved the
	edge-localized spectral Tur\'an inequality
	\[
	\lambda_1(G)^2
	\leq
	2\sum_{e\in E}
	\left(1-\frac1{c(e)}\right),
	\]
	where $c(e)$ denotes the maximum order of a clique containing the
	edge $e$, and they also characterized the corresponding equality
	graphs. Kannan, Kumar, and
	Pragada~\cite{kannan2025localizationspectralturantypetheorems} subsequently initiated a
	systematic study of localized spectral Tur\'an-type inequalities.
	
	The localization program has since been extended to
	the signless Laplacian and the $A_\alpha$-matrix 
	\cite{kannan2026localizedturantypeinequalitiesqindex}, as well as to spectral Tur\'an
	inequalities for signed graphs~\cite{xie2026localization}. The present work
	completes the vertex-localized positive square-energy problem for
	arbitrary graphs and determines all the extremal graphs.
	
	Kannan, Kumar, and Pragada ( \cite{kannan2025localizationspectralturantypetheorems}, Conjecture 1.11) proposed the vertex-localized strengthening
	\begin{equation}\label{eq:localized-conjecture}
		\sqrt{\splus(G)}\leq \sum_{v\in V}\left(1-\frac1{c(v)}\right),
	\end{equation}
	which is stronger than~\eqref{eq:global-LTZ}, because
	$c(v)\leq\omega(G)$ for every vertex.  
	
	The purpose of the present paper is
	to prove~\eqref{eq:localized-conjecture} for every graph and determine all
	equality cases. 
	
	\begin{theorem}[Main theorem]\label{thm:main}
		Let $G$ be a finite simple graph.  Then
		\begin{equation}\label{eq:main-ineq}
			\sqrt{\splus(G)}
			\leq
			\sum_{v\in V}\left(1-\frac1{c(v)}\right).
		\end{equation}
		Equality holds if and only if one of the following occurs:
		\begin{enumerate}[label=\textup{(\roman*)}]
			\item $G$ is edgeless;
			\item for some integers $r\geq2$, $t\geq1$, and $q\geq0$,
			\[
			G\cong K_{\underbrace{t,\ldots,t}_{r\text{ parts}}}\,\dot\cup\,qK_1.
			\]
		\end{enumerate}
		
	\end{theorem}
	
	\subsection*{Outline of the proof}
	We retain two ingredients from Liu et al. and do not reproduce
	their proofs.  The first is their local harmonic inequality for the pivot
	and edge-separation probabilities of the Caro-Wei process.  The second is
	the deterministic upper bound for the separated mass of a doubly
	nonnegative matrix across a partition.

	The new localization is obtained as follows.  Let $K$ be the random pivot
	clique, let $R=|K|$, and put
	\[
	\eta:=\E\left(\frac1R\right).
	\]
	Writing $\tau_v=1/c(v)$, every pivot $v\in K$ satisfies
	$\tau_v\leq1/R$.  Hence
	\[
	\sum_{v\in V}\tau_v^2\Prob(v\in K)
	=
	\E\left(\sum_{v\in K}\tau_v^2\right)
	\leq
	\E\left(\frac1R\right)=\eta.
	\]
	This localizes the harmonic estimate.  At the same time, the doubly nonnegative (DNN)
	partition bound in an outcome with $R$ blocks contains the factor
	$1-1/R$; after taking expectations, it contains the matching factor
	$1-\eta$.  The two estimates fit together through the exact identity
	\[
	(n-D)^2-(1-\eta)\left(n^2-\frac{D^2}{\eta}\right)
	=\frac{(n\eta-D)^2}{\eta},
	\qquad
	D:=\sum_{v\in V}\frac1{c(v)}.
	\]
	This yields a vertex-localized DNN Motzkin-Straus inequality. By applying
	this to $M=A_+\circ A_+$, we prove~\eqref{eq:main-ineq}.
	
	For equality, we first reduce to one connected nontrivial component.
	Equality in the localized coupling forces every pivot in every ordering
	to have local clique number equal to the number of pivots in that
	ordering.  It follows that $c(v)=\omega(G)$ at every vertex and that
	every Caro-Wei outcome has exactly $\omega(G)$ blocks. The equality case in the global bound then settles the equality in the local case.

	\subsection*{Organization}
	Section~\ref{sec:inputs} recalls the Caro-Wei process and the local harmonic inequality from Liu et al.
	Section~\ref{sec:localized-DNN} records another input from Liu et al. and proves the new localized DNN inequality.
	Section~\ref{sec:spectral} gives a proof of the inequality, and
	Section~\ref{sec:equality} determines all equality cases.
	
	\section{  Caro-Wei process}\label{sec:inputs}
	
	The random-order greedy argument underlying the Caro-Wei bound goes back
	to Caro~\cite{caro1979new} and Wei~\cite{wei1981lower}.  We use a formulation of Liu, Tang, and
	Zhang~\cite[Definition~1.5]{liu2026positive}.
	
	\begin{definition}[Caro-Wei process]\label{def:CW}
		Choose a uniformly random total ordering of $V$.  Starting with
		$S_0=V$, repeat the following operation while $S_i\neq\varnothing$.
		Let $x_i$ be the first vertex of $S_i$ in the chosen ordering and set
		\[
		S_{i+1}:=S_i\cap N(x_i),
		\qquad
		B_i:=S_i\setminus S_{i+1}.
		\]
		The vertices $x_0,x_1,\ldots$ are the \emph{pivots}, and the nonempty
		sets $B_0,B_1,\ldots$ are the \emph{blocks}.  The blocks form a partition
		$\mathcal P$ of $V$, and the pivots form a clique.  We denote the
		random pivot set by $K$ and its cardinality by
		\[
		R:=|K|.
		\]
		The number of blocks of $\mathcal P$ is exactly $R$.
	\end{definition}
	
	For $v\in V$ and $uv\in E$, define
	\[
	p_v:=\Prob(v\in K),
	\qquad
	q_{uv}:=\Prob(u\text{ and }v\text{ lie in distinct blocks of }\mathcal P).
	\]
	Note that both probabilities are positive.  The following estimate is the local
	harmonic inequality of Liu et al.  We quote it without
	reproducing the proof.
	
	\begin{theorem}(\cite{liu2026positive},Theorem~2.3)
		\label{thm:LTZ-local-harmonic}
		For every graph $G$ of order $n$ and every $v\in V$,
		\begin{equation}\label{eq:local-harmonic}
			\frac1{p_v}
			+
			\sum_{u\in N(v)}\frac1{q_{uv}}
			\leq n.
		\end{equation}
	\end{theorem}
	
	\begin{corollary}\label{cor:global-harmonic}
		Set
		\begin{equation}\label{eq:B-H-def}
			B:=\sum_{v\in V}\frac1{p_v},
			\qquad
			H:=\sum_{uv\in E}\frac1{q_{uv}}.
		\end{equation}
		Then
		\begin{equation}\label{eq:B-plus-2H}
			B+2H\leq n^2.
		\end{equation}
	\end{corollary}
	
	\begin{proof}
		Sum~\eqref{eq:local-harmonic} over all vertices.  Each edge $uv$ appears
		once in the neighborhood sum at $u$ and once in the neighborhood sum at
		$v$.
	\end{proof}
	
	A real matrix $M$ is positive semidefinite if $M$ is symmetric and $x^{\top}Mx \geq 0$ for all $ x \in \mathbb{R}^n$ and is denoted by $M \succeq 0$. A real symmetric matrix $M$ is \emph{doubly nonnegative} if $M\succeq0$
	and $M$ is entrywise nonnegative.  For a partition
	$\mathcal Q=\{Q_1,\ldots,Q_k\}$ of $V$, define its separated
	$M$-mass by
	\[
	W_{\mathcal Q}(M)
	:=
	\sum_{\substack{uv\in E\\
			u,v\text{ lie in distinct blocks of }\mathcal Q}}
	M_{uv}.
	\]

	\section{The vertex-localized DNN inequality}\label{sec:localized-DNN}
	The following lemma proved in \cite{liu2026positive} will be needed crucially.
	\begin{lemma}\cite{liu2026positive}
		\label{lem:DNN-partition}
		Let $M$ be a doubly nonnegative matrix indexed by $V$, and put
		\begin{equation} \label{eq:T-def}
			T:=\one^{\top}M\one.
		\end{equation}
		If $\mathcal Q=\{Q_1,\ldots,Q_k\}$ has $k\geq1$ nonempty blocks, then
		\begin{equation}\label{eq:DNN-partition}
			W_{\mathcal Q}(M)
			\leq
			\frac12\left(1-\frac1k\right)T.
		\end{equation}
	\end{lemma}
	
	Let $M$ be a doubly nonnegative matrix indexed by $V$. Define the expected separated $M$-mass by
	\begin{equation}\label{eq:Wbar-def}
		\mathcal W
		:=
		\E W_{\mathcal P}(M)
		=
		\sum_{uv\in E}M_{uv}q_{uv}.
	\end{equation}
	\begin{theorem}
		\label{thm:localized-DNN}
		Let $G$ be a finite simple graph and let $M$ be a doubly nonnegative
		matrix indexed by $V$.  Set
		\[
		T:=\one^{\top}M\one,
		\qquad
		S:=\sum_{uv\in E}\sqrt{M_{uv}}.
		\]
		Then
		\begin{equation}\label{eq:localized-DNN}
			2S\leq \sum_{v\in V}\left(1-\frac1{c(v)}\right)\sqrt{T}.
		\end{equation}
	\end{theorem}
	
	\begin{proof}
		Recall that $n:=|V|$ and define
		\begin{equation}\label{eq:tau-D-Phi}
			\tau_v:=\frac1{c(v)},
			\qquad
			D:=\sum_{v\in V}\tau_v.
		\end{equation}
		Run the Caro-Wei process from \Cref{def:CW}, and let
		\begin{equation}\label{eq:eta-def}
			\eta:=\E\left(\frac1R\right).
		\end{equation}
		Since $R\geq1$, we have $\eta>0$.
		
		For every outcome, $K$ is a
		clique of order $R$.  Therefore every pivot $v\in K$ satisfies
		$c(v)\geq R$, and hence
		\[
		\tau_v=\frac1{c(v)}\leq\frac1R.
		\]
		It follows pointwise that
		\[
		\sum_{v\in K}\tau_v^2
		\leq
		R\left(\frac1R\right)^2
		=\frac1R.
		\]
		Taking expectations gives
		\begin{equation}\label{eq:Theta-bound}
			\Theta
			:=
			\sum_{v\in V}\tau_v^2p_v
			=
			\E\left(\sum_{v\in K}\tau_v^2\right)
			\leq\eta.
		\end{equation}
		
		By the Cauchy-Schwarz inequality,
		\begin{align}
			D^2
			&=
			\left(
			\sum_{v\in V}
			\tau_v\sqrt{p_v}\,\frac1{\sqrt{p_v}}
			\right)^2 \notag\\
			&\leq
			\left(\sum_{v\in V}\tau_v^2p_v\right)
			\left(\sum_{v\in V}\frac1{p_v}\right)
			=\Theta B
			\leq\eta B.
			\label{eq:D-square}
		\end{align}
		Consequently,
		\begin{equation}\label{eq:B-lower}
			B\geq\frac{D^2}{\eta}.
		\end{equation}
		Combining \Cref{cor:global-harmonic} with~\eqref{eq:B-lower}, we obtain
		\begin{equation}\label{eq:H-upper}
			H
			\leq
			\frac12\left(n^2-B\right)
			\leq
			\frac12\left(n^2-\frac{D^2}{\eta}\right).
		\end{equation}

		An outcome with $R$ pivots has exactly $R$ blocks.  Applying
		\Cref{lem:DNN-partition} to that outcome gives
		\[
		W_{\mathcal P}(M)
		\leq
		\frac12\left(1-\frac1R\right)T.
		\]
		Taking expectations and using~\eqref{eq:eta-def},
		\begin{equation}\label{eq:Wbar-upper}
			\mathcal W
			\leq
			\frac12(1-\eta)T.
		\end{equation}
		
		Cauchy-Schwarz inequality gives
		\begin{align}
			S^2
			&=
			\left(
			\sum_{uv\in E}
			\sqrt{M_{uv}q_{uv}}\,\frac1{\sqrt{q_{uv}}}
			\right)^2 \notag\\
			&\leq
			\left(\sum_{uv\in E}M_{uv}q_{uv}\right)
			\left(\sum_{uv\in E}\frac1{q_{uv}}\right)
			=\mathcal W H.
			\label{eq:weighted-CS}
		\end{align}
		Substituting~\eqref{eq:H-upper} and~\eqref{eq:Wbar-upper} into
		\eqref{eq:weighted-CS},
		\begin{equation}\label{eq:S-intermediate}
			S^2
			\leq
			\frac{T}{4}(1-\eta)
			\left(n^2-\frac{D^2}{\eta}\right).
		\end{equation}
		
		The identity
		\begin{align}
			(n-D)^2
			&-
			(1-\eta)
			\left(n^2-\frac{D^2}{\eta}\right) \notag\\
			&=(n-D)^2-(1-\eta)n^2+(1-\eta)\frac{D^2}{\eta} \notag\\
			&=\eta n^2-2nD+\frac{D^2}{\eta} \notag\\
			&=\frac{(n\eta-D)^2}{\eta}
			\geq0
			\label{eq:perfect-square}
		\end{align}
		shows that
		\[
		(1-\eta)
		\left(n^2-\frac{D^2}{\eta}\right)
		\leq(n-D)^2.
		\]
		Hence~\eqref{eq:S-intermediate} yields
		\[
		S^2\leq\frac{(n-D)^2T}{4}.
		\]
		Since $S,T\geq0$, taking square roots gives
		\[
		2S\leq(n-D)\sqrt{T}=\sum_{v\in V}\left(1-\frac1{c(v)}\right)\sqrt{T},
		\]
		which is~\eqref{eq:localized-DNN}.
	\end{proof}

	\section{Proof of the inequality}\label{sec:spectral}
	
	Let
	\begin{equation}\label{eq:spectral-decomp}
		A=\Aplus-\Aminus,
		\qquad
		\Aplus,\Aminus\succeq0,
		\qquad
		\Aplus\Aminus=\Aminus\Aplus=0
	\end{equation}
	be the positive--negative spectral decomposition of $A=A(G)$.
	
	\begin{proof}[Proof of the inequality in \Cref{thm:main}]
		Set
		\begin{equation}\label{eq:spectral-M}
			X:=\Aplus,
			\qquad
			M:=X\circ X.
		\end{equation}
		The matrix $M$ is entrywise nonnegative.  Since $X\succeq0$, the Schur
		product theorem~\cite[Theorem~7.5.3]{MR2978290} implies $M\succeq0$;
		hence $M$ is doubly nonnegative.  With the notation of
		\Cref{thm:localized-DNN},
		\begin{align}
			T
			&=\one^{\top}M\one
			=\sum_{u,v\in V}X_{uv}^2
			=\norm{X}_F^2 \notag\\
			&=\tr(X^2)
			=\sum_{\lambda_i(G)>0}\lambda_i(G)^2
			=\splus(G).
			\label{eq:T-splus}
		\end{align}
		Moreover, using~\eqref{eq:spectral-decomp},
		\begin{align}
			\tr(AX)
			&=\tr\bigl((\Aplus-\Aminus)\Aplus\bigr) \notag\\
			&=\tr(\Aplus^2)-\tr(\Aminus\Aplus)
			=\tr(\Aplus^2)
			=T.
			\label{eq:T-trAX}
		\end{align}
		Also,
		\begin{equation}\label{eq:trAX-edge}
			\tr(AX)=2\sum_{uv\in E}X_{uv}.
		\end{equation}
		Therefore, with
		\[
		S:=\sum_{uv\in E}\sqrt{M_{uv}}
		=\sum_{uv\in E}|X_{uv}|,
		\]
		\eqref{eq:T-trAX} and~\eqref{eq:trAX-edge} give
		\begin{equation}\label{eq:T-le-2S}
			T
			=2\sum_{uv\in E}X_{uv}
			\leq
			2\sum_{uv\in E}|X_{uv}|
			=2S.
		\end{equation}
		Applying \Cref{thm:localized-DNN} to $M$ and using~\eqref{eq:T-splus},
		\[
		T\leq2S\leq\sum_{v\in V}\left(1-\frac1{c(v)}\right)\sqrt{T}.
		\]
		If $T=0$, then~\eqref{eq:main-ineq} is immediate.  If $T>0$, division by
		$\sqrt{T}$ gives
		\[
		\sqrt{\splus(G)}=\sqrt{T}\leq\sum_{v\in V}\left(1-\frac1{c(v)}\right). \qedhere
		\]
	\end{proof}
	
	\section{Extremal graphs}\label{sec:equality}
	
	\subsection{Reduction to one nontrivial component}
	
	\begin{lemma}\label{lem:component-reduction}
		Suppose equality holds in~\eqref{eq:main-ineq}.  Then either $G$ is
		edgeless, or $G$ has exactly one connected component containing an edge;
		all other components are isolated vertices.  In the latter case, equality
		holds for the unique nontrivial component.
	\end{lemma}
	
	\begin{proof}
		Let $G_1,\ldots,G_m$ be the connected components of $G$, and put
		\[
		s_i:=\splus(G_i),
		\qquad
		\Phi_i:=\sum\limits_{w \in V(G_i)}\left(1-\frac{1}{c(w)}\right).
		\]
		The adjacency matrix is block diagonal, and for every $v \in V$, $c(v)$ is
		computed within components.  Hence
		\[
		\splus(G)=\sum_{i=1}^m s_i,
		\qquad
		\sum_{v\in V}\left(1-\frac1{c(v)}\right)=\sum_i \Phi_i.
		\]
		Applying the already proved inequality to each component,
		$s_i\leq\Phi_i^2$.  Therefore
		\begin{equation}\label{eq:component-chain}
			\sqrt{\splus(G)}
			=\sqrt{\sum_i s_i}
			\leq\sqrt{\sum_i\Phi_i^2}
			\leq\sum_i\Phi_i
			=\sum_{v\in V}\left(1-\frac1{c(v)}\right).
		\end{equation}
		For nonnegative $\Phi_i$, equality in the second inequality is equivalent
		to
		\[
		\sum_{i<j}\Phi_i\Phi_j=0.
		\]
		Thus at most one $\Phi_i$ is positive.  A connected component has
		$\Phi_i=0$ exactly when it consists of a single isolated vertex. Thus, every
		vertex in a nontrivial connected component satisfies $c_{G_i}(v)\geq2$.
		Hence at most one component contains an edge.  If such a component exists,
		equality throughout~\eqref{eq:component-chain} also forces equality for
		that component.
	\end{proof}
	
	For the remainder of the necessity proof, assume that $G$ is connected,
	has at least one edge, and satisfies
	\begin{equation}\label{eq:connected-equality}
		\sqrt{\splus(G)}=\sum_{v\in V}\left(1-\frac1{c(v)}\right).
	\end{equation}
	In particular, $n\geq2$ and every vertex has positive degree.
	
	Set $X=\Aplus$, $M=X\circ X$, and retain
	\begin{equation}\label{eq:T-S-special}
		T:=\one^{\top}M\one=\splus(G),
		\qquad
		S:=\sum_{uv\in E}\sqrt{M_{uv}}.
	\end{equation}
	By~\eqref{eq:T-le-2S}, \Cref{thm:localized-DNN}, and
	\eqref{eq:connected-equality},
	\[
	T\leq2S\leq\sum_{v\in V}\left(1-\frac1{c(v)}\right)\sqrt{T}=T.
	\]
	Thus
	\begin{equation}\label{eq:DNN-equality-special}
		2S=\sum_{v\in V}\left(1-\frac1{c(v)}\right)\sqrt{T}.
	\end{equation}

	\subsection{Equality forces constant local clique number}
	
	\begin{lemma}
		
		\label{lem:local-clique-rigidity}
		Under~\eqref{eq:connected-equality},
		\[
		c(v)=\omega(G)
		\qquad\text{for every }v\in V.
		\]
		Moreover, every Caro-Wei ordering has exactly $\omega(G)$ pivots.
	\end{lemma}
	
	\begin{proof}
		Retain all notation from the proof of \Cref{thm:localized-DNN}:
		\[
		\tau_v=\frac1{c(v)},
		\quad
		D=\sum_v\tau_v,
		\quad
		\eta=\E\left(\frac1R\right),
		\]
		\[
		B=\sum_v\frac1{p_v},
		\quad
		H=\sum_{uv\in E}\frac1{q_{uv}},
		\quad
		\Theta=\sum_v\tau_v^2p_v,
		\quad
		\mathcal W=\sum_{uv\in E}M_{uv}q_{uv}.
		\]
		Define
		\begin{equation}\label{eq:upper-H-W}
			\widehat H
			:=\frac12\left(n^2-\frac{D^2}{\eta}\right),
			\qquad
			\widehat{\mathcal W}
			:=\frac12(1-\eta)T.
		\end{equation}
		The proof of \Cref{thm:localized-DNN} gives
		\begin{equation}\label{eq:equality-long-chain}
			S^2
			\leq\mathcal W H
			\leq\widehat{\mathcal W}\,\widehat H
			\leq\frac{(n-D)^2T}{4}.
		\end{equation}
		By~\eqref{eq:DNN-equality-special}, the first and last terms in
		\eqref{eq:equality-long-chain} are equal.
		
		Because $G$ is connected and nonempty, every first pivot has a neighbor,
		so every outcome has $R\geq2$.  Hence $\eta<1$ and
		$\widehat{\mathcal W}>0$.  Also $H>0$, since $E\neq\varnothing$.
		Since
		\[
		0\leq\mathcal W\leq\widehat{\mathcal W},
		\qquad
		0<H\leq\widehat H,
		\]
		equality of the endpoint products in~\eqref{eq:equality-long-chain}
		forces
		\begin{equation}\label{eq:W-H-equality}
			H=\widehat H.
		\end{equation}
		
		The proof of~\eqref{eq:H-upper} gives the more detailed chain
		\[
		H
		\leq\frac12(n^2-B)
		\leq\frac12\left(n^2-\frac{D^2}{\eta}\right)
		=\widehat H.
		\]
		Together with~\eqref{eq:W-H-equality}, this yields
		\begin{equation}\label{eq:B-exact}
			B=\frac{D^2}{\eta}.
		\end{equation}
		Now~\eqref{eq:D-square}, \eqref{eq:Theta-bound}, and
		\eqref{eq:B-exact} imply
		\[
		D^2\leq\Theta B\leq\eta B=D^2.
		\]
		Since $B>0$, equality holds throughout and therefore
		\begin{equation}\label{eq:Theta-eta}
			\Theta=\eta.
		\end{equation}
		
		For a fixed total ordering $\pi$, let $K_\pi$ be its pivot set and
		$R_\pi:=|K_\pi|$.  Pointwise,
		\begin{equation}\label{eq:pointwise-tau}
			\sum_{v\in K_\pi}\tau_v^2\leq\frac1{R_\pi}.
		\end{equation}
		The expectation of the difference between the two sides is
		$\eta-\Theta=0$ by~\eqref{eq:Theta-eta}.  Every ordering has positive
		probability, and the difference in~\eqref{eq:pointwise-tau} is
		nonnegative.  Hence equality holds in~\eqref{eq:pointwise-tau} for every
		ordering.
		
		For each pivot $v\in K_\pi$, we have $\tau_v\leq1/R_\pi$.  There are
		$R_\pi$ pivots, and
		\[
		\sum_{v\in K_\pi}\tau_v^2
		=\frac1{R_\pi}
		=R_\pi\left(\frac1{R_\pi}\right)^2.
		\]
		Thus we must have
		\begin{equation}\label{eq:c-pivot-R}
			c(v)=R_\pi
			\qquad
			(v\in K_\pi)
		\end{equation}
		for every ordering $\pi$.
		
		Let $uv\in E$.  Choose an ordering whose first two vertices are $u$
		and $v$, in that order.  Then $u$ is the first pivot and $v$ is the
		second pivot, because $v\in N(u)$ and is first in the induced order on
		$N(u)$.  Applying~\eqref{eq:c-pivot-R} to both pivots gives
		$c(u)=c(v)$.  Since $G$ is connected, $c(v)$ is constant over all
		vertices.  A vertex belonging to a maximum clique has local clique number
		$\omega(G)$, so
		\[
		c(v)=\omega(G)
		\qquad(v\in V).
		\]
		Finally, in every ordering the first pivot satisfies
		$c(x_0)=R_\pi$ by~\eqref{eq:c-pivot-R}; hence
		$R_\pi=\omega(G)$ for every ordering.
	\end{proof}
	
	\subsection{Connected necessity case}
	
	The equality-attaining graphs in the global bound \eqref{eq:global-LTZ} were determined in \cite{jayarajan2026equalitycasepositivesquareenergy}.

	\begin{lemma}
		Under \eqref{eq:connected-equality}, $\omega(G) \mid n$ and  $G\cong K_{\underbrace{n/\omega(G),\ldots,n/\omega(G)}_{\omega(G)\text{ parts}}}$.
	\end{lemma}
	\begin{proof}
		Applying \Cref{lem:local-clique-rigidity} and using \eqref{eq:connected-equality}, we obtain
		\begin{equation*}
			\sqrt{\splus(G)}=\bigg(1-\dfrac{1}{\omega(G)}\bigg)n.
		\end{equation*}
		Therefore, by \Cref{thm;globaleq}, $\omega(G) \mid n$ and $G\cong K_{n/\omega(G),\ldots,n/\omega(G)}$.
		
	\end{proof}
	
	\subsection{Verification of the extremal graphs}
	
	\begin{proof}[Proof of \Cref{thm:main} (equality)]
		By \Cref{lem:component-reduction}, every equality-attaining graph containing an edge has one
		connected nontrivial component $G_1$, which the preceding argument shows to be
		$K_{t,\ldots,t}$ with $t=\dfrac{\vert V(G_1) \vert}{\omega(G_1)}$ and $r\geq2$ equal parts; every other component is an
		isolated vertex.  Thus necessity gives
		\[
		G\cong K_{\underbrace{t,\ldots,t}_{r\text{ parts}}}\,\dot\cup\,qK_1
		\qquad(r\geq2,\ t\geq1,\ q\geq0).
		\]
		
		Conversely, every edgeless graph has $\splus(G)=0$ and $c(v)=1$ for
		all $v$, so both sides of~\eqref{eq:main-ineq} vanish.  Now let
		\[
		G=K_{\underbrace{t,\ldots,t}_{r\text{ parts}}}\,\dot\cup\,qK_1,
		\qquad r\geq2,\quad t\geq1,\quad q\geq0.
		\]
		The complete $r$-partite component is regular of degree
		$(r-1)t$ and has exactly one positive adjacency eigenvalue (see Theorem 6.7 in \cite{MR572262}).  The isolated vertices
		contribute only zero eigenvalues.  Hence
		\[
		\sqrt{\splus(G)}=(r-1)t.
		\]
		Every vertex in the nontrivial component has local clique number $r$,
		whereas every isolated vertex has local clique number $1$.  Therefore
		\[
		\sum_{v\in V}\left(1-\frac1{c(v)}\right)
		=rt\left(1-\frac1r\right)+q(1-1)
		=(r-1)t
		=\sqrt{\splus(G)}.
		\]
		This proves sufficiency and completes the proof of \Cref{thm:main}.
	\end{proof}

	\section*{Acknowledgments}
	Abhay Jayarajan acknowledges support from the DST INSPIRE Fellowship (IF220692). M.~Rajesh Kannan acknowledges financial support from ANRF-CRG, India (File No.\ CRG/2023/002747). Rahul Roy thanks the University Grants Commission (UGC), India, for financial support (NTA Ref. No. 231610209574).

	\section*{Declaration of AI use}
	The authors acknowledge the use of ChatGPT (GPT-5.6, OpenAI; accessed August 2026) for improving the exposition of the manuscript and refining some of the proofs. All mathematical statements have been independently verified by the authors, who are solely responsible for the correctness of the results presented in this paper.
	
	\bibliographystyle{plainnat}
	\bibliography{Ref_loc_edited}

	\affl{Abhay Jayarajan}{ma23resch02001@iith.ac.in, abhayjayarajan@gmail.com}{ Department of Mathematics, Indian Institute of Technology Hyderabad, Kandi, Sangareddy 502284, India.}
	
	\affl{M. Rajesh Kannan}{rajeshkannan@math.iith.ac.in, rajeshkannan1.m@gmail.com}{Department of Mathematics, Indian Institute of Technology Hyderabad, Kandi, Sangareddy 502284, India.}
	
	\affl{Shivaramakrishna Pragada}{shivaramakrishna\_pragada@sfu.ca, shivaramkratos@gmail.com}{Department of Mathematics, Simon Fraser University, Burnaby, Canada}
	
	\affl{Rahul Roy}{ma23resch11004@iith.ac.in, rahulroy13832@gmail.com}{ Department of Mathematics, Indian Institute of Technology Hyderabad, Kandi, Sangareddy 502284, India.}
\end{document}